\documentclass[11pt,leqno]{amsart}
\usepackage{amsmath, amssymb, amsthm, stmaryrd}
\usepackage{euscript}
\usepackage{verbatim}
\usepackage{graphicx}

\newcommand{\chapter}{\section}
\newcommand{\Mbar}{\overline{M}} 
\newcommand{\Mdel}{\partial M}
\newcommand{\Minf}{M(\infty)}

\DeclareMathOperator{\grad}{grad}

\DeclareMathOperator{\Rc}{Rc}

\DeclareMathOperator{\Rm}{Rm}

\newcommand{\bR}{\mathbb{R}}

\newcommand{\bN}{\mathbb{N}}
\newcommand{\bS}{\mathbb{S}}
\newcommand{\bT}{\mathbb{T}}
\newcommand{\bO}{\mathbb{O}}
\newcommand{\bU}{\mathbb{U}}

\newcommand{\cA}{ {\mathcal A} }

\newcommand{\cK}{ {\mathcal K} }

\newcommand{\gbar}{\overline{g}}
\newcommand{\rhob}{\overline{\rho}}
\newcommand{\Gambar}{\overline{\Gamma}}
\newcommand{\Gamtil}{\widetilde{\Gamma}}
\newcommand{\gtil}{\widetilde{g}}
\newcommand{\grd}{\mathring{g}}
\newcommand{\grdi}{ (\mathring{g}_i) }

\newcommand{\pd}[2]{\frac{\partial #1}{\partial #2}}
\newtheorem{theorem}{Theorem}

\newtheorem{cor}[theorem]{Corollary}
\newtheorem{lemma}[theorem]{Lemma}
\numberwithin{equation}{section}

\begin{document}
\title[AH metrics]{Intrinsic characterization for Lipschitz asymptotically hyperbolic metrics}
\author[Eric Bahuaud]{Eric Bahuaud}
\date{\today}
\address{Current address: \newline
Institut de Math\'ematiques et de Mod\'elisation de Montpellier \newline
UMR 5149 CNRS - Universit\'e Montpellier II \newline
Case Courrier 051 - Place Eug\`ene Bataillon \newline
34095 Montpellier, France \newline }
\email{ebahuaud@math.univ-montp2.fr}
\begin{abstract}
Conformally compact asymptotically hyperbolic metrics have been intensively studied.  The goal of this note is to understand what intrinsic conditions on a complete Riemannian manifold $(M,g)$ will ensure that $g$ is AH in this sense.  We use the geodesic compactification by asymptotic geodesic rays to compactify $M$ and appropriate curvature decay conditions to study the regularity of the conformal compactification.  
\end{abstract}
\maketitle

\section{Introduction}
\label{section:introduction}
Conformally compact metrics provide a good model of asymptotically hyperbolic (AH) geometry.  These metrics have been intensively studied for over 20 years and are important in Riemannian and conformal geometry, and in physics.  See \cite{Anderson} or \cite{Lee} and the extensive bibliographies therein for more details.  The definition of these metrics however is given extrinsically.  In this note we study to what extent intrinsic conditions can determine an AH metric.  We begin by reviewing the usual setting.

Suppose $(M,g)$ is a noncompact Riemannian $(n+1)$-manifold that is the interior of a compact manifold with boundary $\Mbar$.  For $k \in \bN_0, \alpha \in [0,1]$, the metric $g$ is \textit{$C^{k,\alpha}$ conformally compact} if there exists a defining function $\rho$ for the boundary such that $\gbar = \rho^2 g$ extends to be a $C^{k,\alpha}$ metric on $\Mbar$.  Such a metric induces a conformal class on the boundary $\partial M$, called the \textit{conformal infinity} of $g$.

Straightforward calculations yield that if $g$ is at least $C^2$ conformally compact then the sectional curvatures in $M$ satisfy
\begin{equation} \label{AHdecay} 
\sec = -|d\rho|^2_{\gbar} + O(\rho). 
\end{equation}
If $|d\rho|^2_{\gbar} = 1$ on $\Mdel$, then the sectional curvatures of $M$ approach $-1$ near $\Mdel$.  This justifies the following definition.  The metric $g$ is \textit{asymptotically hyperbolic} if $g$ is conformally compact and $|d\rho|^2_{\gbar} = 1$ on $\Mdel$.  The classical setting typically requires at least a $C^2$ conformal compactification.  As any two defining functions for $\Mdel$ differ by a multiplication by a positive function, this definition is easily seen to be independent of $\rho$.  When $g$ is additionally an Einstein metric, i.e. $\Rc g = - n g$, the sectional curvatures of $g$ satisfy an improved decay estimate, i.e.
\begin{equation} \label{AHEdecay} 
\sec = -1 + O(\rho^2). 
\end{equation}

The natural question that lingers in one's mind is to what extent conformally compact AH metrics can be characterized intrinsically.  We present intrinsic conditions on a complete Riemannian manifold $(M,g)$ that will ensure that $g$ is at least Lipschitz conformally compact AH (see Theorem \ref{thm:B} below).  We use the geodesic compactification (we relegate definitions to the next section) by asymptotic geodesic rays to compactify $M$, and appropriate curvature decay conditions to study the regularity of the conformal compactification.  In a different direction there are several rigidity results for AH metrics assuming faster curvature decay than we allow; see \cite{ShiTian} and the references therein for more details.

Our main result is
\begin{theorem} \label{thm:B} Suppose $(M,g)$ is a complete noncompact Riemannian manifold and $K$ is an essential subset (see pg. \pageref{pg:es}).  Let $r(x) = dist_g(x, K)$.  Assume further that
\begin{equation} \label{convexity} \tag{NSC}
\sec( \overline{ M \backslash K } ) < 0, \text{and}
\end{equation}
\begin{equation} \label{AH1}  \tag{AH1}
\sec( M \backslash K ) = -1 + O( e^{-r} ), \text{ and }
\end{equation} 
\begin{equation} \label{AH2}  \tag{AH2+}
|\nabla_g \Rm|_g = O( e^{-\omega r} ), \text{for some} \; \omega > 1.
\end{equation}
Then $\Mbar = M \cup M(\infty)$ is a topological manifold with boundary endowed with a $C^{1,1}$ structure independent of $K$.  Further $\gbar := e^{-2r} g$ extends to a $C^{0,1}$ metric on $\Mbar$, i.e. $g$ is $C^{0,1}$ conformally compact.
\end{theorem}

We use a result on regularity of the geometric compactification given sectional curvature bounds from \cite{BahuaudMarsh} to compactify $M$ and obtain the first estimate of manifold regularity.  All of the assumptions above certainly hold sufficiently close to the boundary of a smoothly conformally compact Einstein metric.  Of course \eqref{AH1} is just the decay of sectional curvature like \eqref{AHdecay} expressed in terms of the intrinsic distance $r$.  This assumption leads to an estimate for the Riccati equation for the shape operator of hypersurfaces of constant $r$ value, and the $r$-derivative of the metric satisfies a linear equation involving the shape operator in certain coordinates.  In order to have Lipschitz control of the compactified metric $\gbar$ we will require control of the second coordinate derivatives of $g$; this means we require some control on the derivatives of the Riccati equation.  Unfortunately the assuming the rate of decay of $|\nabla_g Rm|$ expected from a generic smoothly conformally compact AH metric is not enough to guarantee a Lipschitz conformal compactification.  Assumption \eqref{AH2} provides the required control.  

We remark \label{remark:curv} that assumption \eqref{AH2} implies that sectional curvature enjoys the same rate of decay, i.e. in fact $\sec( M \backslash K ) = -1 + O( e^{-\omega r} )$.  This is easily seen by integrating the components of $R+\cK$ (here $\cK$ is the constant curvature tensor) with respect to a parallel frame along normal geodesics emanating from $K$.  This argument works initially assuming only $\sec( M \backslash K ) = -1 + o(1)$, and consequently Theorem 1 holds with this kind of sectional curvature decay.

The author and Romain Gicquaud plan a more detailed study of how curvature decay hypotheses can be used to prove regularity for conformal compactifications.  In particular we are able to considerably strengthen the results presented here when $g$ is an Einstein metric.

This note is organized as follows.  In Section \ref{section:background}, we recall the geodesic compactification and fix notation.  In Section \ref{section:lipmfld}, we prove that $\Mbar$ is a $C^{0,1}$ manifold, and explain how metric estimates lead to a subsequent improvement in manifold regularity.  In Section \ref{sec:lipcc}, we prove Theorem \ref{thm:B}.  Finally in Section \ref{section:example}, we provide an example of a metric satisfying the decay of sectional curvature and the first covariant derivative of curvature like that of a generic AH metric but with no Lipschitz conformal compactification using our techniques.

This work formed part of the author's Ph.D. dissertation \cite{Bahuaud}.  It is a pleasure to thank Jack Lee for guidance and inspiration throughout.  I am also indebted to Robin Graham and Marc Herzlich for useful conversations and suggestions.  I thank Yuguang Shi for pointing out a calculation error in an early preprint of this paper.

\section{The Geodesic compactification}
\label{section:background}
 
The geodesic compactification of simply connected manifolds of negative curvature originates in the work of Eberlein and O'Neill \cite{EberleinONeill}.  Two geodesic rays $\sigma$ and $\tau$ parameterized by arc-length are called \textit{asymptotic} if $d_M(\sigma(t), \tau(t))$ remains bounded as $t \rightarrow +\infty$.  Denote the set of equivalence classes by $M(\infty)$.  Eberlein and O'Neill proved that given $p \in M$, $M(\infty)$ is in bijection with the unit sphere $S_p \subset T_p M$ by a rescaled exponential map, and that there is a natural topology and smooth structure that makes the \textit{geodesic compactification} $\Mbar := M \cup M(\infty)$, diffeomorphic to a closed ball.  In general one expects only the topological structure to be independent of $p$.

In 1985, Anderson and Schoen \cite{AndersonSchoen} proved if $(M,g)$ is a simply connected manifold with pinched negative sectional curvature like 
$$-\infty < -b^2 \leq \sec(M) \leq -a^2 < 0,$$
then $M(\infty)$ has a $C^{0,\alpha}$ structure, where $\alpha = a/b$.  However, they did not prove regularity for the geodesic compactification $\Mbar = M \cup M(\infty)$, and they used the strong hypothesis of simple connectivity.  

In \cite{BahuaudMarsh}, Tracey Marsh and I extended and generalized this result.  For the remainder of this note let $(M,g)$ be a smooth noncompact complete Riemannian manifold, and let $K \subset M$ be a compact embedded smooth $(n+1)$-dimensional submanifold with boundary such that $Y := \partial K$ is convex with respect to the outward pointing unit normal vector field.  We say that $K$ is an \textit{essential subset}\label{pg:es} for $M$ if additionally the outward normal exponential map, $E$, from $Y$ is a diffeomorphism $E: Y \times [0,\infty) \rightarrow \overline{M \backslash K}$.  

In \cite{BahuaudMarsh} we provided one sufficient condition for an essential subset; if $K$ is totally convex (see \cite{BahuaudMarsh} for a definition) and $\sec(M \backslash K) < 0$, then $K$ is an essential subset.  The exponential map again extends to $\Mbar$.  By declaring this map to be a diffeomorphism we obtain a topology and smooth structure on $\Mbar$ depending on $K$.  We may now study to what extent these structures depend on $K$.  Assume further curvature pinching analogous to that used by Anderson and Schoen: 
\begin{equation} \label{pinchedcurvature}
 -\infty < -b^2 \leq \sec( M \backslash K ) \leq -a^2 < 0.
\end{equation}
The main theorem from \cite{BahuaudMarsh} proved that given \eqref{pinchedcurvature}, the topology induced on $\Mbar$ is independent of $K$ and $\Mbar$ is endowed with a $C^{0,\alpha}$ structure independent of $K$, where $\alpha = a/b$.

The proofs of these results use the techniques of comparison geometry.  In the present paper we will obtain metric estimates from our curvature assumptions tailored to the AH case and appeal to Theorem 17 of \cite{BahuaudMarsh} in order to obtain a $C^{0,1}$ manifold structure for $\Mbar$ independent of $K$.  After obtaining appropriate metric estimates, we are able to prove Theorem \ref{thm:B}.

We now fix notation.  For consistency we follow the conventions established in \cite{BahuaudMarsh}.  If $X, Z$ are orthonormal vectors, the sectional curvature of the plane they span is given by $\sec(X,Z) = \Rm(X,Z,Z,X)$, where $\Rm$ is the Riemannian curvature 4-tensor.  We will denote this tensor by $R$ in the sequel.  We let $\cK$ denote the tensor of constant curvature $+1$, i.e. $\cK_{ijkl} = g_{il} g_{jk} - g_{ik} g_{jl}$.

Given that $X,Z$ are vector fields on $Y$ extended arbitrarily in $M$, we define the second fundamental form of $Y$ by $h(X,Z) = g( \nabla_X Z, -\partial_r )$, where $\nabla$ is the connection of $g$ in $M$.  We say $Y$ is strictly convex if $h$ is positive definite.  We define the shape operator as the $(1,1)$ tensor $S$ characterized by $g(X, S(Z)) = h(X,Z)$ for all $X, Z$ as above.

In what follows an inequality involving the shape operator of the form $S \geq c$  means that every eigenvalue of $S$ is greater than or equal to $c$.  Inequalities involving a metric are to be interpreted as inequalities between quadratic forms.

In light of the diffeomorphism $Y \times [0, \infty) \approx \overline{M \backslash K}$ and the fact that $r$ is the distance to $K$, we may decompose $g$ as
\[ g = dr^2 + g_Y(y,r), \]
where $g_Y$ is a one parameter family of metrics on $Y$.  Following \cite{BahuaudMarsh} we choose coordinates $\{y^{\beta}\}$ on sufficiently small open sets $W \subset Y$ and extend $\{y^{\beta}\}$ to be constant along integral curves of $\grad_g r$ to obtain \textit{Fermi coordinates} $( y^{\beta}, r) $ on $\overline{M \backslash K}$, and a decomposition of the metric as
\[ g = dr^2 + g_{\beta \nu}(y,r) dy^{\beta} dy^{\nu}. \]

We use Latin indices to index directions in $M$ and consequently these indices range from $0$ to $n$.  \label{page:greek}With the exception of the letter $\rho$, we use Greek indices to index directions along $Y$ and these range from $1$ to $n$, and $0$ to index the direction normal to $Y$.  The letter $\rho$ which next appears in Section \ref{sec:lipcc} is already well established as a defining function for the boundary.

\section{Lipschitz compactification of $\Mbar$}
\label{section:lipmfld}

In this section we obtain our first set of metric and shape operator estimates and use them to obtain a first (manifold) compactification of $M$.

Given an essential subset $K$ with boundary $Y = \partial K$, a reference covering for $Y$ is a (finite) covering $\{ W_i\}$ \label{choosingW} by small open balls in $Y$ with sufficiently small radius chosen so that $g_{\beta \nu}$ (transfered to $W$ by means of normal coordinates), the round metric $\grd$ on $\bS^n$ in normal coordinates and a flat metric on $W$ are all comparable.  In particular, on cylinders of the form $W_i \times [r_0, \infty)$, we have a comparison hyperbolic metric $dr^2 + \sinh(r) \grd_i$.

In order to prove our first compactification result, we cite Theorem 17 of \cite{BahuaudMarsh}.  We include it here for convenience.  
\begin{theorem} \label{thm:holdercomp} Let $M$ be a complete Riemannian manifold containing an essential subset.  Suppose for every essential subset $K \subset M$ with reference covering $\{W_i\}$ for $Y=\partial K$ that there exists an $R>0$ such that for every $r>R$
\begin{eqnarray*}
\frac{ \sinh^2( a( r - R ) )} {a^2} \; \grdi_{\beta \nu} \; \leq g_{\beta \nu}(y,r) \; \leq \;\frac{ \sinh^2( b( r+ R ) ) }{b^2}\; \grdi_{\beta \nu},
\end{eqnarray*}
for all i.  Then $\Mbar$ has a $C^{0,a/b}$ structure independent of $K$.
\end{theorem}

Consequently the focus of this section is to prove a metric comparison estimate on cylinders $W \times [r_0, \infty)$, where $W \subset Y$ is a sufficiently small ball. 

We begin with a lemma.  Note that in what follows we use calculus for Lipschitz functions; differential equations and inequalities are to be interpreted in the almost everywhere sense.

\begin{lemma} \label{lemma:decayest}Suppose that $f$ is a bounded function of $r > 0$ such that there exists a constant $J>0$ with
\[ |f(r) - 1| \leq J e^{-r}. \]
Suppose further that $\lambda$ is a positive Lipschitz solution of the Riccati equation
\begin{equation} \label{scalarriccati}
\begin{aligned}
\lambda' + \lambda^2 &= f(r), \text{and } \nonumber \\
\lambda(0) & > 0.
\end{aligned}
\end{equation}
Then there exists a positive constant $C = C(J, \lambda(0))$ such that
\[ | \lambda - 1 | \leq C e^{-r}, \]
for all $r$ sufficiently large.
\end{lemma}
\begin{proof}
It is easy to argue from the Riccati equation that $\lambda$ is bounded above by some constant: 
\[ \lambda \leq c_1. \]

The next refinement follows the idea of \cite[Lemma 2.3]{ShiTian}.

Set $v = \lambda - 1$.  Then $-1 < v < \lambda$ and
\begin{align*}
(v^2)' + 2v^2 & \leq (v^2)' + (4 + 2v) v^2 \\
& = 2v v' + (4+2v) v^2 \\
&= 2v( \lambda' + \lambda^2 - 1)\\
&\leq c_2 e^{-r},
\end{align*}
where $c_2$ is a constant depending on $c_1$ and $J$.  Note that the last inequality uses the lower bound for $f$ in case that both $v$ and the expression involving $\lambda$ are negative.  Consequently
\[ (v^2 e^{2r})' \leq c_2 e^r, \]
and upon integrating we find that 
\[ v^2 e^{2r} \leq c_2 e^r + c_3, \]
and so
\begin{equation} \label{eqn:shivdecay}
 |v| \leq c e^{-r/2}. 
\end{equation}

We now consider a further refinement for $\lambda$.  
Choose a constant $K > \max\{J, \lambda(0)\}$ and consider the test function $u = 1 + K e^{-r}$.  Then
\[ u' + u^2 = 1 + K e^{-r} + K^2 e^{-2r}.\]
Define $F(r, z) = 1 + K e^{-r} + K^2 e^{-2r} - z^2$, and observe that $F$ is continuous.  Now
\[ \lambda' \leq 1 + J e^{-r} - \lambda^2 \leq F(r, \lambda).\]
We now compare $\lambda$ and $u$.  Note that in our eventual application we will only need to consider the case where $\lambda - u > 0$.  Observe that in this case
\begin{align} \label{Flip}
F(r, \lambda) - F(r, u) &= u^2 - \lambda^2 \nonumber \\
&= (u+\lambda)(u-\lambda) \\
&= -(u+\lambda)(\lambda-u) \nonumber \\  
&\leq -(\lambda - u), \nonumber
\end{align}
because $ 0 < \lambda < c_1$ and $1 \leq u \leq 1+K$ implies $-(u+\lambda) < -1$.

By construction of $K$, $\lambda(0) < u(0)$.  We now claim $\lambda(r) \leq u(r) = 1 + K e^{-r}$.  The rest of the argument follows as in the proof of Theorem 1.11.7 of \cite{BirkoffRota}; we repeat it here.  If by way of contradiction there is some $T>0$ where $\lambda(T) > u(T)$, then set
\[ r^* = \sup \{ 0 < r < T: \lambda(r) \leq u(r) \}. \]
By continuity, $\lambda(r^*) = u(r^*)$, and $\lambda-u > 0$ on $(r^*,T]$.  Estimate \eqref{Flip} applies and on this interval
\[ (\lambda - u)' \leq F(r, \lambda) - F(r, u) \leq - (\lambda-u). \]
This inequality easily implies $\lambda - u \leq 0$ on $(r^*, T]$, a contradiction.  We conclude that
\[ \lambda \leq u \leq 1 + K e^{-r}.\]

We now consider the refinement for the lower bound for $\lambda$.  By estimate \eqref{eqn:shivdecay} above, choose $r_0$ so large that $\lambda > 1/2$ for $r \geq r_0$.

Consider the test function $u = 1 - K e^{-r}$, for some $K$ to be determined.  Observe that
\[ u' + u^2 = 1 - Ke^{-r} + K^2 e^{-2r} = 1 - K e^{-r} u. \]
Set $F(r, z) = 1 - Ke^{-r} z - z^2$.  We have
\[ \lambda' + \lambda^2 \geq 1 - J e^{-r}.\]
and we would like to determine $K$ so that $\lambda' \geq F(r, \lambda)$.  We compute
\[ \lambda' \geq 1 - J e^{-r} - \lambda^2 \geq 1 - 2 J \lambda e^{-r} - \lambda^2, \]
at least for $r \geq r_0$.  We find that $\lambda' \geq F(r, \lambda)$ is implied by the condition that $K \geq 2J$.  Finally we have
\[ u(r_0) = 1 - K e^{-r_{0}}, \]
so we choose $K$ so large that $K \geq 2J$ and $u(r_0) < 1/2 < \lambda(r_0)$.  We also have a Lipschitz estimate
\begin{align*}
F(r, \lambda) - F(r, u) &= u^2 - \lambda^2 + K e^{-r}( u- \lambda) \\
&= (u+\lambda + 2K e^{-r}) (u - \lambda) \\
&\geq c(K) (u - \lambda),
\end{align*}
when $u > \lambda$.

We claim $\lambda \geq u$ on $[r_0, \infty)$.  Following the same type of argument as before, if by way of contradiction there is a $T$ where $u(T) > \lambda(T)$ we again restrict to an interval $(r^*, T]$ where $u > \lambda$.  But then
\[ \lambda' - u' \geq F(r,\lambda) - F(r,u) \geq -c(K) (\lambda-u), \]
which again implies $\lambda \geq u$ on $[r^*, T]$, a contradiction.  We conclude $\lambda > 1 - Ke^{-r}$ for $r$ sufficiently large.

\end{proof}

We now apply this lemma to our geometric situation to obtain improved estimates for the shape operator and metric.

\begin{theorem}[Comparison theorem]\label{thm:comparison}  Given curvature assumption \eqref{AH1} and convexity assumption \eqref{convexity}, let $(y^{\beta}, r)$ be Fermi coordinates for $Y$ on $W \times [0,\infty)$ for an open set $W \subset Y$.  Let $\Lambda, \lambda$ denote the maximum and minimum eigenvalues of the shape operator over $W$, and let $\Omega, \omega$ denote the maximum and minimum eigenvalue of the metric over $W$ (taken with respect to the background euclidean metric).  There exist positive constants $C, L_1$ and $L_2$ depending on these eigenvalues such that for $r$ sufficiently large we have
\flushleft\textbf{ Shape operator estimate: } 
\begin{eqnarray*}
\label{eqn:shapeinequality}
 (1 - C e^{- r})  \; \delta^{\beta}_{\gamma} \; \leq \; S^{\beta}_{\gamma}(y, r) \; \leq \; (1 + C e^{-r}) \; \delta^{\beta}_{\gamma},
\end{eqnarray*}
\textbf{ Metric estimate: }
\begin{eqnarray*}
L_1 e^{2r} \; \delta_{\beta \nu} \leq g_{\beta \nu}(y, r) \leq L_2 \; e^{2r} \; \delta_{\beta \nu}.
\end{eqnarray*}
\label{lemma:metriccomp}
\end{theorem}
\begin{proof} We sketch the broad idea here, for a complete proof see \cite{Bahuaud}.  
Throughout the proof, primes ($'$) denote derivatives with respect to $r$.  We suppress the dependence on the tangential variable.

In Fermi coordinates the shape operator, $S(r)$, of $r$-level sets satisfies a Riccati equation:
$$ ( S'(r) + S^2(r) )_{\beta}^{\nu} = -R_N(r)_{\beta}^{\nu},$$
where $R_N$ is the normal curvature operator given by $g( R_N V, V) = \sec(V, \partial_r)$, when $V$ is a $g$-unit vector.  Let $\lambda_M(r)$ be the maximum eigenvalue of $S(r)$.  Then $\lambda_M$ is Lipschitz continuous.  In virtue of \eqref{AH1}, we have the estimate
\begin{equation} \label{eqn:lamest}
 \lambda_M' + \lambda_M^2 = 1 + O(e^{-r}),
\end{equation}
with a similar estimate holding for the minimum eigenvalue of shape, $\lambda_m(r)$ as well.

We are only assuming that $Y = \partial K$ is convex.  However analysis of the Riccati equation for the shape operator and \eqref{convexity} easily imply that by possibly modifying the essential subset, the eigenvalues of the shape operator are uniformly bounded from below by a positive constant.  Consequently we ensure $\lambda_M \geq \lambda_m > 0$.  By Lemma \ref{lemma:decayest} we conclude that $\lambda_M$ satisfies the estimate
$$ \lambda_M \leq 1 + C e^{- r},$$
for some positive constant $C$.  This implies
\[ S^{\beta}_{\gamma}(y, r) \; \leq \; (1 + C e^{-r}) \; \delta^{\beta}_{\gamma}. \]
Similar analysis for the minimum eigenvalue gives the lower bound for the shape operator for $r$ sufficiently large.  

Metric estimates follow from similar type of analysis for the equation
$$ \partial_r g_{\beta \nu}(y,r) = 2 S^{\gamma}_{\beta}(y,r) g_{\gamma \nu}(y,r). $$
\end{proof}

\begin{cor}[Hyperbolic Metric comparison]\label{cor:comparison}  Consider Fermi coordinates $(y^{\beta}, r)$ for $Y$ on $W\times [0, \infty)$.  There exists an $R$ depending on $W$ such that for every $r > R$:
\label{thm:hcomp}
\begin{eqnarray}
\sinh^2( ( r - R ) ) \; \grd_{\beta \nu} \; \leq g_{\beta \nu}(y,r) \; \leq \; \sinh^2(  r+ R ) \; \grd_{\beta \nu}.
\end{eqnarray}
\end{cor}
\begin{proof}
Observe that it is possible to choose $R$ so large that $L_2 e^{2r} \leq \sinh^2(r + R)$ and $\sinh^2(r-R) \leq L_1 e^{2r}$, for $r > R$.  
\end{proof}

We now come to the main result of this section.

\begin{theorem} \label{thm:lipmlfd} Let $(M^{n+1},g)$ be a complete noncompact Riemannian manifold.  Suppose that there exists an essential subset $K$ for $M$ and that curvature assumptions \eqref{convexity} and \eqref{AH1} are satisfied.  Then $\Mbar := M \cup \Minf$ is endowed with the structure of a $C^{0,1}$ manifold with boundary independent of the choice of $K$.
\end{theorem}
\begin{proof}
Given a reference covering for $Y$, Corollary \ref{cor:comparison} indicates that $g$ is comparable to an upper and lower hyperbolic comparison metric.  Consequently we may apply Theorem \ref{thm:holdercomp} in the case $a=b=1$.
\end{proof}

\subsection{Metric estimates and regularity for $\Mbar$}
\label{sec:metricimproved}

If $\cA$ is a $C^{l,\beta}$ atlas for a manifold, the transformation formula for the metric under a change of coordinates shows that in general the metric is well-defined up to $C^{l-1, \beta}$ regularity.  We now prove a lemma that allows us to use metric regularity to improve the regularity of transition functions.  This lemma is in the spirit of \cite{CalabiHartman}; see that paper for further results along these lines.

\begin{lemma} \label{lemma:improvedregularity} Suppose $M$ is a manifold with two smooth atlases $\cA_1 = \{(U_{\alpha}, \phi_{\alpha})\}$ and $\cA_2 = \{(V_{\beta}, \psi_{\beta})\}$ that are $C^{0,1}$ compatible.  Suppose that $g$ is a metric that is $C^{0,1}$ with respect to both atlases.  Then $\cA_1$ and $\cA_2$ are $C^{1,1}$ compatible.
\end{lemma}
\begin{proof}
This is a local question so we reduce to the case where $f: (U \subset \bR^{n+1}, x^i) \rightarrow (\widetilde{U} \subset \bR^{n+1}, y^i)$ is a $C^{0,1}$ diffeomorphism between open sets of $\bR^{n+1}$.  
Write 
\[ g_{ij} = g\left( \pd{}{x^i}, \pd{}{x^j} \right) \; \mbox{and} \; \gtil_{kl} = g\left( \pd{}{y^k} , \pd{}{y^l} \right). \]
As $g$ satisfies Lipschitz estimates in both systems of coordinates, Christoffel symbols are defined a.e. and bounded.  The transformation law for Christoffel symbols under a change of coordinates states
\[ \pd{^2 y^m}{x^i \partial x^j} = \pd{y^k}{x^i} \pd{y^l}{x^j} \Gamma^m_{kl} - \left( \Gamtil^l_{ij} \circ f \right) \pd{y^m}{x^l}. \]
The right hand side of this equation is bounded if $f \in C^{0,1}$ and if $g \in C^{0,1}$ with respect to both sets of coordinates.   Consequently $f$ satisfies a $C^{1,1}$ estimate.
\end{proof}

\section{Lipschitz Conformal Compactification} \label{sec:lipcc}
In this section we show that $\gbar = e^{-2r} g$ extends to a Lipschitz metric on $\Mbar$ and prove Theorem \ref{thm:B}.  Fix an essential subset $K$.  In Fermi coordinates on $W \times [0,\infty)$, where $W$ is a sufficiently small open ball (see page \pageref{choosingW}), we may write
\[ g = dr^2 + g_{\alpha \beta}(y,r) dy^{\alpha} dy^{\beta}. \]
We set $\rho := e^{-r}$, and remind the reader of the convention given on page \pageref{page:greek} that $\rho$ does not count as a tangential or Greek variable. In these `compactified Fermi coordinates', $(y^{\beta},\rho)$ over $W \times (0, 1]$, we now have
\[ g = \frac{d\rho^2}{\rho^2} + g_{\alpha \beta}(y, - \log \rho) dy^{\alpha} dy^{\beta}, \]
and consequently
\[ \gbar = d\rho^2 + \rho^2 g_{\alpha \beta}(y, - \log \rho) dy^{\alpha} dy^{\beta}. \]
In order to prove a Lipschitz estimate for $\gbar$ it suffices to prove that the first derivatives of $\gbar$ are bounded in compactified Fermi coordinates.  These derivatives are:
\begin{equation*}
\begin{aligned}
\partial_{\rho} \gbar_{\alpha \beta} &= 2 \rho g_{\alpha \beta} + \rho^2 \partial_r g_{\alpha \beta} \cdot \left(-\frac{1}{\rho}\right) \\ 
& = 2 \rho^{-1} (\delta^{\gamma}_{\alpha}-S^{\gamma}_{\alpha}) \gbar_{\gamma \beta}, \\
\partial_{\mu} \gbar_{\alpha \beta} &= \rho^2 \partial_{\mu} g_{\alpha \beta}.
\end{aligned}
\end{equation*}

As we explain below, the tangential derivatives of $\gbar$ satisfy a system of differential equations.  ODE comparison theory will then be used to obtain $L^{\infty}$ estimates for the derivatives.

A few words about how we measure the size of tensors is in order.  The metric estimate from Theorem \ref{lemma:metriccomp} is an inequality between quadratic forms.  However the polarization identity implies that each component of $g$ with respect to a Fermi coordinate basis is also of order $e^{2r}$.  Further, if we are provided with decay estimates for the eigenvalues of any $(1,1)$-tensor $T$ that is self-adjoint with respect to $g$ then we may also obtain decay estimates for each component of $T$.  We exploit these facts for both the shape operator $S$ and normal curvature $R_{0 \alpha \; 0}^{ \;\;\; \beta}$ below.

By the shape operator estimate \eqref{eqn:shapeinequality}, the $\rho$ derivative of $\gbar$ is bounded.  In order to estimate tangential derivatives of $g$ we take tangential derivatives of the Riccati and metric equations
\begin{equation} \label{eqn:tansys} 
\left\{ \begin{aligned}
 (\partial_{\mu} S_{\alpha}^{\beta})' &= -(\partial_{\mu} S^{\beta}_{\gamma}) S^{\gamma}_{\alpha} - S^{\beta}_{\gamma} (\partial_{\mu} S^{\gamma}_{\alpha}) -\partial_{\mu} R_{0 \alpha \; 0}^{\; \; \; \; \beta}, \\
 (\partial_{\mu} g_{\alpha \beta})' &= 2 (\partial_{\mu} S^{\gamma}_{\alpha}) g_{\gamma \beta} + 2 S^{\gamma}_{\alpha} (\partial_{\mu} g_{\gamma \beta}),
\end{aligned}
\right.
\end{equation}
where primes denote a coordinate derivative with respect to $r$.  In the following analysis we will regard this as a first order system of $2n^3$ linear differential equations by introducing new dependent variables.  In order to use the intrinsic decay estimate \eqref{AH2}, we must express the coordinate derivative of curvature that appears in the system \eqref{eqn:tansys} in terms of a covariant derivative.  In particular we have
\begin{lemma} \label{lemma:curdecay} Given curvature decay assumptions \eqref{AH1} and \eqref{AH2},
\begin{equation}
-\partial_{\mu} R_{0 \alpha \; 0}^{\; \; \; \; \beta} = - \Gamma^{\sigma}_{\mu \alpha} ( \delta_{\sigma}^{\beta} + R_{0 \sigma \; 0}^{\; \; \; \; \beta}) + \Gamma^{\beta}_{\mu \sigma} (\delta^{\sigma}_{\alpha} + R_{0 \alpha \; 0}^{\; \; \; \; \sigma}) + F,
\end{equation}
where $F$ is a remainder term whose components satisfy $F = O( e^{\Omega r})$, where $\Omega = 1 - \omega$.
\end{lemma}
\begin{proof}
Take the $\mu$-covariant derivative of the curvature tensor $R_{0 \alpha \; 0}^{\; \; \; \; \beta}$ to obtain
\begin{equation} \label{eqn:covrm}
\nabla_{\mu} R_{0 \alpha \; 0}^{\; \; \; \beta} = \partial_{\mu} R_{0 \alpha \; 0}^{\; \; \; \beta} - \Gamma^{\sigma}_{\mu 0} R_{\sigma \alpha \; 0}^{\; \; \; \; \beta}  - \Gamma^{\sigma}_{\mu \alpha} R_{0 \sigma \; 0}^{\; \; \; \; \beta} + \Gamma^{\beta}_{\mu \sigma} R_{0 \alpha \; 0}^{\; \; \; \; \sigma} - \Gamma^{\sigma}_{\mu 0} R_{0 \alpha \; \sigma}^{\; \; \; \; \beta}.
\end{equation}
We begin by considering the second and fifth terms above.  In Fermi coordinates $-\Gamma^{\sigma}_{\mu 0} = S^{\sigma}_{\mu}$.  Further, by the remark on page \pageref{remark:curv}, we have $|R + \cK|_g = O(e^{- \omega r})$ (recall $\cK$ is the constant curvature 4-tensor).  Since $R_{\sigma \alpha \; 0}^{\; \; \; \; \beta}$ contains one $0$ index it is straightforward to verify that
\begin{equation} \label{eqn:covrm1}
|\Gamma^{\sigma}_{\mu 0} R_{\sigma \alpha \; 0}^{\; \; \; \; \beta} + \Gamma^{\sigma}_{\mu 0} R_{0 \alpha \; \sigma}^{\; \; \; \; \beta} | = O(e^{(1-\omega)r}).
\end{equation}
Set $F = \Gamma^{\sigma}_{\mu 0} R_{\sigma \alpha \; 0}^{\; \; \; \; \beta} + \Gamma^{\sigma}_{\mu 0} R_{0 \alpha \; \sigma}^{\; \; \; \; \beta} $.
We now consider the third and fourth terms.  Observe
\begin{equation} \label{eqn:covrm2}
-\Gamma^{\sigma}_{\mu \alpha} R_{0 \sigma \; 0}^{\; \; \; \; \beta} + \Gamma^{\beta}_{\mu \sigma} R_{0 \alpha \; 0}^{\; \; \; \; \sigma} = - \Gamma^{\sigma}_{\mu \alpha} ( R_{0 \sigma \; 0}^{\; \; \; \; \beta} + \delta_{\sigma}^{\beta} ) + \Gamma^{\beta}_{\mu \sigma} (R_{0 \alpha \; 0}^{\; \; \; \; \sigma} + \delta^{\sigma}_{\alpha}).
\end{equation}
Solving \eqref{eqn:covrm} for the coordinate derivative of curvature and applying estimates \eqref{AH2}, \eqref{eqn:covrm1} and equation \eqref{eqn:covrm2} we obtain
\begin{equation*}
-\partial_{\mu} R_{0 \alpha \; 0}^{\; \; \; \; \beta} = -  \Gamma^{\sigma}_{\mu \alpha} ( R_{0 \sigma \; 0}^{\; \; \; \; \beta} + \delta_{\sigma}^{\beta} ) + \Gamma^{\beta}_{\mu \sigma} (R_{0 \alpha \; 0}^{\; \; \; \; \sigma} + \delta^{\sigma}_{\alpha}) + F,
\end{equation*}
where we absorb the covariant derivative term into $F$.  Note the components of $F$ are $O(e^{ (1-\omega) r})$ by \eqref{AH2}, as we consider the $\mu$-covariant derivative $\nabla_{\mu} R_{0 \alpha \; 0}^{\; \; \; \beta} dx^{\alpha} \otimes \partial_{\beta}= (\nabla R)(\partial_0, \cdot, \cdot, \partial_0, \partial_{\mu})$ as  a $(1,1)$-tensor and $|\partial_0|_g = 1$, $|\partial_{\mu}|_g = O(e^r)$, by Theorem \ref{thm:comparison}.
\end{proof}
\label{page:decayremark}

We also make a change of variables.  Set $W = e^{2r} S$ and recall $\gbar = e^{-2r} g$.  Note that for Greek indices, $\Gambar^{\gamma}_{\alpha \beta} = \Gamma^{\gamma}_{\alpha \beta}$.  Using this change of variables and Lemma \ref{lemma:curdecay}, system \eqref{eqn:tansys} becomes
\begin{equation} \label{eqn:tansys2} 
\left\{ \begin{aligned}
 (\partial_{\mu} W^{\beta}_{\alpha})' & = - (\partial_{\mu} W^{\beta}_{\gamma}) S^{\gamma}_{\alpha} - S^{\beta}_{\gamma} (\partial_{\mu} W^{\gamma}_{\alpha}) + 2 \partial_{\mu} W^{\beta}_{\alpha} \\  & - e^{2r} \Gambar^{\sigma}_{\mu \alpha} ( R_{0 \sigma \; 0}^{\; \; \; \; \beta} +  \delta_{\sigma}^{\beta}) + e^{2r} \Gambar^{\beta}_{\mu \sigma} (R_{0 \alpha \; 0}^{\; \; \; \; \sigma} + \delta^{\sigma}_{\alpha}) + e^{2r} F_{\alpha \;}^{\; \beta}, \\
(\partial_{\mu} \gbar_{\alpha \beta})' & = 2 e^{-2r} (\partial_{\mu} W^{\gamma}_{\alpha}) \gbar_{\gamma \beta} + 2 S^{\gamma}_{\alpha} (\partial_{\mu} \gbar_{\gamma \beta}) - 2 (\partial_{\mu} \gbar_{\alpha \beta}).
\end{aligned}
\right.
\end{equation}
Regard the components $\partial_{\mu} W^{\beta}_{\alpha}$ and $\partial_{\mu} \gbar_{\alpha \beta}$ as vectors in $\bR^{n^3}$ which we denote $\partial W$ and $\partial \gbar$.  The above system may be compactly written as:
\begin{equation} \label{eqn:tansys3} 
\left\{ \begin{aligned}
(\partial W)' &= A \partial W + B \partial \gbar + G, \\
(\partial \gbar)' &= C \partial W + D \partial \gbar.
\end{aligned}
\right.
\end{equation}
where $A,B,C,D, G$ are $(n^3 \times n^3)$-matrices.  We will not need the explicit form of these matrices in what follows; we only need estimates on the size of the matrix entries.  Let $|\cdot|_2$ denote the usual euclidean norm in $\bR^{n^3}$ (we use the same symbol for the associated operator norm on matrices).  We have
\begin{lemma} \label{lemma:coefest} There exists a positive constant $c$ such that the coefficients of the system \eqref{eqn:tansys3} satisfy the following estimates
\begin{align*}
|A|_2 & \leq c e^{-r}, \\
|B|_2 & \leq c e^{r}, \\
|C|_2 & \leq c e^{-2r}, \\
|D|_2 & \leq c e^{-r}, \\
|G|_2 & \leq c e^{(2+\Omega)r}. 
\end{align*}
\end{lemma}
\begin{proof}
The lemma follows from the curvature assumptions and the metric and shape operator estimates of Theorem \ref{thm:comparison}.  For example to analyze the nonzero components of $A$, consider the obvious interpolation
\begin{align*} - (\partial_{\mu} W^{\beta}_{\gamma}) S^{\gamma}_{\alpha} &- S^{\beta}_{\gamma} (\partial_{\mu} W^{\gamma}_{\alpha}) + 2 \partial_{\mu} W^{\beta}_{\alpha} \\  &= - (\partial_{\mu} W^{\beta}_{\gamma}) (S^{\gamma}_{\alpha} - \delta^{\gamma}_{\alpha} ) - (S^{\beta}_{\gamma} - \delta^{\beta}_{\gamma}) (\partial_{\mu} W^{\gamma}_{\alpha}). \end{align*}
By the shape operator estimate, the eigenvalues of $S - I$ are order $e^{-r}$.  Consequently,
\[ |A|_2 \leq c e^{-r}. \]
Similarly with the other matrices.
\end{proof}

We now finish the proof of the main theorem.

\begin{proof}[Proof of Theorem \ref{thm:B}]
In order to prove that the tangential derivatives of $\gbar$ are bounded, we first compare system \eqref{eqn:tansys3} to a model system.

Set $x(r)= |\partial W|$ and $y(r) = |\partial \gbar|$.  These functions are continuous everywhere and are smooth where $\partial W$ and $\partial \gbar$ are nonzero.  Note that the system uncouples if either $x=0$ or $y=0$ on an interval, and from there it is easy to see that $y$ is bounded, which implies $\gbar$ is Lipschitz on that interval.  From the Cauchy-Schwarz inequality it follows that where they are smooth, $x'(r) \leq |(\partial W)'|$ and $y'(r) \leq |(\partial \gbar) '|$.  Consequently the system of equations \eqref{eqn:tansys3} combined with Lemma \ref{lemma:coefest} implies the estimate
\begin{equation} \label{eqn:tansys4} 
\left\{ \begin{aligned}
x' & \leq c e^{-r}x + c e^{ r} y + c e^{(2+\Omega)r}, \\
y' & \leq c e^{-2r}  x + c e^{-r} y.
\end{aligned}
\right.
\end{equation}

We thus compare solutions to this system of differential inequalities to solutions of the corresponding system of equations
\begin{equation} \label{eqn:modelsys} 
\left\{ \begin{aligned}
u' & = c e^{-r} u + c e^{r} v + c e^{(2+\Omega)r}, \\
v' & = c e^{-2r} u + c e^{-r} v.
\end{aligned}
\right.
\end{equation}

We now digress to discuss the solutions of this model system.  The system \eqref{eqn:modelsys} may be solved for $u$ and $u'$ in terms of $v$ and $v'$, obtaining
\[ u = \frac{1}{c} e^{2r} v' - e^r v, \; \text{and } \]
\[ u'= e^r v' + c (e^r - 1) v + c e^{(2+\Omega)r}. \]
This leads to a second order linear inhomogeneous differential equation for $v$,
\begin{equation}
\label{eqn:secforv}
v'' + (2 - 2ce^{-r}) v' + c e^{-r} (c e^{-r} - c - 1) v = c^2 e^{\Omega r}. 
\end{equation}
This equation has asymptotically constant coefficients.  By Theorem 1.9.1 of \cite{Eastham}, two linearly independent solutions $v_1$ and $v_2$ of the associated homogeneous equation to \eqref{eqn:secforv} satisfy asymptotic estimates
\[ v_1 = (1 + o(1) ) e^{-2r}, \; v_1' = (-2 + o(1)) e^{-2r}, \mbox{and} \]
\[ v_2 = 1 + o(1), \; v_2' = o(1).\]

We must also obtain estimates for the solution of the inhomogeneous part of \eqref{eqn:secforv}.  Set $p(r) = 2 - 2ce^{-r}$ and $q(r) = c e^{-r}(c e^{-r} - c -1 )$.  Equation \eqref{eqn:secforv} may be transformed to a first order system by letting $w =v'$:
\begin{equation*}
\begin{pmatrix} v \\ w 
\end{pmatrix}'
= \begin{pmatrix} 0 & 1 \\ -q & -p 
\end{pmatrix} 
\begin{pmatrix} v \\ w 
\end{pmatrix}
+
\begin{pmatrix} 0 \\ c^2 e^{\Omega r}
\end{pmatrix}
\end{equation*}

The fundamental matrix for this system is 
\[ \phi = \begin{pmatrix} v_1 & v_2 \\ v_1' & v_2' 
\end{pmatrix}, \]
and it is easy to verify that by `variation of parameters', a solution to the inhomogeneous system is given by
\[ \begin{pmatrix} v(r) \\ w(r) 
\end{pmatrix}_{IH} = \phi(r) \int_{r_0}^r { \phi^{-1}( s ) \begin{pmatrix} 0 \\ c e^{\Omega s} \end{pmatrix} \; ds}. \]
Note that $\det \phi = (-2 + o(1)) e^{-2r}$.  It is now straightforward to check that a solution to the inhomogeneous equation, $v_{IH}$, satisfies $v_{IH} = O(e^{\Omega r})$.

We repeat this process for $u$ and find that $u$ satisfies
\begin{equation} \label{eqn:secforu}
 u'' + (-1 - 2c e^{-r})u' + ( c(2-c) e^{-r} + c^2 e^{-2r}) u = (\Omega - 1)c e^{(2+\Omega)r}-c^2 e^{(1+\Omega)r}.
\end{equation}
Linearly independent solutions are given by
\[ u_1 = 1 + o(1), \; u_1' = o(1), \; \mbox{ and } \]
\[ u_2 = (1 + o(1)) e^r, \; u_2' = (1+o(1)) e^r. \]
A solution to the inhomogeneous equation satisfies $u_{IH} = O( e^{(2+\Omega)r} )$.

Because $\Omega < 0$, the analysis above implies that every solution of \eqref{eqn:secforv} is at worst $O(1)$, and every solution of \eqref{eqn:secforu} satisfies $u = O(e^{(2+\Omega)r})$.  This implies like estimates for the components of every solution of \eqref{eqn:modelsys}.

At a point $r_0$ where $x(r_0) > 0$ and $y(r_0) > 0$, fix initial conditions $u(r_0) > x(r_0)$ and $v(r_0) > y(r_0)$, and consider the solution to \eqref{eqn:modelsys}.  It is straightforward to argue that $u$ and $v$ must remain positive on $[r_0, \infty)$: if by way of contradiction $u = 0$ or $v= 0$ for some $r > 0$ we may compute the first such time
\[ r_u = \inf \{ r > 0: u(r) = 0 \}, \mbox{and} \]
\[ r_v = \inf \{ r > 0: v(r) = 0 \}.\]
We have that $r_u > 0$ and $u > 0$ on $[r_0, r_u)$ as well as $r_v > 0$ and $v > 0$ on $[r_0, r_v)$.  If $r_u \leq r_v$ then the fundamental theorem of calculus implies
\[ u(r) = u(r_0) + \int_{r_0}^r u'(s) ds = u(r_0) + \int_{r_0}^r c e^{-s} u + c e^{s} v + c e^{(2+\Omega)s} ds \]
for all $r \in [r_0, r_u)$.  As every term in the integrand is strictly positive on $[r_0, r_u)$ we find $u(r_u) > 0$, a contradiction.  A similar argument holds for $r_v < r_u$ using the equation for $v'$. 

Theorem \ref{thm:ODEcomp} (cf. the Appendix below) now implies that $|\partial_{\mu} W_{\alpha}^{\beta}| = x \leq u$ and $|\partial_{\mu} \gbar_{\alpha \beta}| = y \leq v$ on $[r_0, \infty)$.  

As $|\partial_{\mu} \gbar_{\alpha \beta}|=O(1)$, the first derivatives of $\gbar$ satisfy Lipschitz estimates in the interior of a compactified Fermi coordinate chart.  Therefore $\gbar$ extends to a Lipschitz continuous 2-tensor up to $\partial M$.  Further, the metric estimate from Theorem \ref{thm:comparison} implies that this extension is positive definite.  This proves that $g$ is $C^{0,1}$ conformally compact.

That $\Mbar$ is a $C^{1,1}$ manifold follows from Lemma \ref{lemma:improvedregularity}.
\end{proof}

\subsection{Appendix: Differential Equation Comparison}

The purpose of this appendix is to prove a comparison theorem for a two-dimensional system of first-order linear differential equations with positive coefficient functions.

\begin{theorem} \label{thm:ODEcomp} Suppose that $a, b, c, d$ are positive smooth functions on $[t_0, t_1]$ (respectively $[t_0, \infty)$).  Suppose that $x$ and $y$ are nonnegative continuous functions that are smooth where they are nonzero and satisfy the differential inequalities
\begin{align*}
x' & \leq a x + by + e, \\
y' & \leq c x + dy + f.
\end{align*}
Suppose in addition that $u$ and $v$ are positive smooth solutions of the corresponding system of differential equations:
\begin{align*}
u' & = a u + bv + e, \\
v' & = c u + dv + f.
\end{align*}
If $x(t_0) < u(t_0)$ and $y(t_0) < v(t_0)$ then $x < u$ and $y < v$ on  $[t_0, t_1]$ (respectively $[t_0, \infty)$).
\end{theorem}

\begin{proof}
The main observation we need is that when $x$ and $y$ are nonzero,
\begin{align*}
(u-x)' &\geq a( u - x) + b( v - y ), \; \; \mbox{and } \\ 
(v-y)' &\geq c( u - x) + d( v - y).
\end{align*}
We know initially $u(t_0) > x(t_0)$ and $v(t_0) > y(t_0)$.  If there is a point where $x$ crosses $u$ or where $y$ crosses $v$ then there must be a first time where this occurs after $t_0$.  This is to say that we can find an $s > t_0$ where $u > x$ and $v > y$ on $[t_0, s)$ but either $u(s) = x(s)$ or $v(s) = y(s)$ or both.  Note that we may assume that $x$ and $y$ are positive on $[t_0,s]$, otherwise we shrink to a neighbourhood of the form $[t_0^*, s)$ where this is the case.

In any of these cases, by the observation above and since $a, b, c, d$ are all positive we have
\[ (u-x)'  \geq 0 \; \; \mbox{and } \; \; (v-y)' \geq 0, \]
on the interval $[t_0, s]$.  This now easily implies that
\[ u(t) - x(t) \geq u(t_0) - x(t_0) > 0, \mbox{and} \]
\[ v(t) - y(t) \geq v(t_0) - y(t_0) > 0. \]
for all $t \in [t_0, s]$, which contradicts the choice of $s$ as the time where either $u(s) = x(s)$ or $v(s) = y(s)$.
\end{proof}

Note there is an obvious generalization of Theorem \ref{thm:ODEcomp} to higher dimensional systems.  Also the sign of the diagonal elements is unrestricted if one writes the systems using an integrating factor.
\section{An Example} \label{section:example}

Recall that a generic smoothly conformally compact AH metric satisfies decay condition \eqref{AH1} and 
\begin{equation} \label{AH2e} \tag{AH2}
|\nabla_g \Rm|_g = O(\rho). 
\end{equation}
In this section we provide an example of a metric satisfying \eqref{AH1}, \eqref{AH2e} on an end diffeomorphic to $\bT^n \times (0,1)$ that does not have a Lipschitz conformal compactification when one uses a geodesic defining function (i.e. one for which $|d\rho|^2_{\gbar} \equiv 1$) to compactify the metric.

We begin by defining a useful function.  Let $\eta: \bR \rightarrow [-1,1]$ be a function such that $\eta|_{[1, \infty)} = 1$, $\eta|_{[-\infty, -1)} = -1$, $\eta|_{[-1/2,1/2]} = x$ such that $\eta$ is smooth, nondecreasing and $\eta' \leq C_1$, $|\eta''| < C_2$, $|\eta'''| < C_3$.  On the open set $\bO = \{ (y, \rho), 0 < \rho < 1, y \in \bR \}$ we consider the function
\[ f(y, \rho) = 2 \int^{1}_{\rho} \eta\left(\frac{ \sin y}{s}\right) ds. \]
Since $\eta$ is bounded, $f = O(1)$.  Computing first derivatives we find
\[ \partial_{\rho} f(y, \rho) = -2 \eta\left(\frac{\sin y}{\rho}\right) = O(1). \]
\[ \partial_{y} f(y, \rho) = 2 \int_{\rho}^{1} \frac{\cos y}{s} \eta'\left(\frac{\sin y}{s}\right) ds = O( \log \rho ). \]
We now compute second derivatives.  Note that when $|\sin y| > \rho$, $\eta'(\sin y/ \rho) = 0$, and when $|\sin y| \leq \rho$, $\eta'(\sin y/ \rho) \leq C_1$.  This reasoning lets us estimate expressions of the form $\eta'( \sin y / \rho) (\sin y / \rho)$ in what follows and similarly for higher derivatives.
\[ \partial_{\rho} \partial_{\rho} f(y, \rho) = 2 \eta'\left(\frac{\sin y}{\rho}\right) \frac{\sin y}{\rho^2} = O(\rho^{-1}). \]
\[ \partial_{y} \partial_{\rho} f(y, \rho) = \partial_{\rho} \partial_{y} f(y, \rho) = -2 \frac{\cos y}{\rho} \eta'\left(\frac{\sin y}{\rho}\right) ds = O( \rho^{-1}). \]
\[ \partial_{y} \partial_{y} f(y, \rho) = 2 \int_{\rho}^{1} \frac{\cos^2 y}{s^2} \eta''\left(\frac{\sin y}{s}\right) - \frac{\sin y}{s^2}  \eta'\left(\frac{\sin y}{s}\right)  ds = O( \rho^{-1} ). \]
A similar computation shows all of the third derivatives of $f$ satisfy $\partial^3 f = O(\rho^{-2})$.

We now construct a metric on $\bU = \{ (y^{\alpha}, \rho), 0 < \rho < 1, y^{\alpha} \in \bR \}$.  Fix a tangential coordinate $y^{\alpha_0}$ and label it as $y$.  Consider the metric
\[ \gbar = d \rho^2 + e^{f(y, \rho) } \delta_{\alpha \beta} dy^{\alpha} dy^{\beta}.\]
We regard this as a metric on the product of a torus with an interval $\bT^n \times [0,1)$.  This metric does not satisfy a uniform Lipschitz estimate down to $\rho = 0$ as 
\begin{equation} \label{eqn:gy}
\partial_y \gbar_{\alpha \beta} = e^{f(y, \rho)} \cdot \partial_y f \cdot \delta_{\alpha \beta} = O(\log \rho).
\end{equation}
Note however that
\begin{equation} \label{eqn:grho}
\partial_{\rho} \gbar_{\alpha \beta} = e^{f(y, \rho)} \cdot \partial_{\rho} f  \cdot \delta_{\alpha \beta} = O(1),
\end{equation}
and where these derivatives appear in the expressions for curvature plays an important role in what follows.
We now study the asymptotic curvature properties of the blow-up metric
\[ g := \rho^{-2} \gbar. \]
This is most easily done using the transformation formula for curvature under a conformal change of metric.  From the expressions for the metrics and bounds on $f$, we obtain estimates for the derivatives of both metrics.  We can translate this into `worst case estimates' for the Christoffel symbols (we suppress indices to indicate any choice of indices is valid)
\begin{equation} \label{eqn:wc} 
\left\{ \begin{aligned}
 \Gamma &= O(\rho^{-1}), \\
 \Gambar &= O(\log \rho), \partial \Gambar = O(\rho^{-1}), \partial^2 \Gambar = O(\rho^{-2}).
\end{aligned}
\right.
\end{equation}
The transformation formula for the full Riemannian curvature tensor under a conformal change of metric using the fact that $|d\rho|^2_{\gbar} \equiv 1$ may be written as
\begin{equation*}
 (R + \cK)_{ijkl} = \rho^{-3} ( \overline{\nabla}^2 \rho \owedge \gbar)_{ijkl} + \rho^{-2} \overline{R}_{ijkl},
\end{equation*}
where $\owedge$ is the Kulkarni-Nomizu product, $\overline{\nabla}$ is the $\gbar$-covariant derivative and $\cK$ is the constant curvature tensor.  This formula is valid for $\rho > 0$.  We now estimate.  Estimates \eqref{eqn:wc} already imply that $\overline{R}_{ijkl} = O(\rho^{-1})$.  We must argue that no log terms occur in $A := \overline{\nabla}^2 \rho \owedge \gbar$.  Observe that 
\[ (\overline{\nabla}^2 \rho \owedge \gbar)_{ijkl} = \rhob_{il} \gbar_{jk} -  \rhob_{ik} \gbar_{jl} + \rhob_{jk} \gbar_{il} - \rhob_{jl} \gbar_{ik}. \]
If all of the indices are distinct then $A = 0$ since the metric is diagonal.  So if $A \neq 0$ for a choice of indices then a pair of indices is repeated.  We now have two subcases whether or not the remaining indices are identical or different.  In case they are identical then by the symmetries of curvature tensors, up to sign we have
\begin{align*}
(\overline{\nabla}^2 \rho \owedge \gbar)_{ijji} &= \rhob_{ii} \gbar_{jj} -  \rhob_{ij} \gbar_{ji} + \rhob_{jj} \gbar_{ii} - \rhob_{ji} \gbar_{ij} \\
&= \rhob_{ii} \gbar_{jj} + \rhob_{jj} \gbar_{ii}.
\end{align*}
The formula for second covariant derivatives of functions applied to the coordinate function $\rho$ implies
\[ \rhob_{ii} = \partial_i \partial_i \rho - \Gambar^s_{ii} \partial_s \rho = - \Gambar^0_{ii} = \frac{1}{2} \partial_{\rho} \gbar_{ii} = O(1), \]
using \eqref{eqn:grho}.
In the second case, by the symmetries of curvature we have up to sign
\begin{align*}
(\overline{\nabla}^2 \rho \owedge \gbar)_{ijjl} &= \rhob_{il} \gbar_{jj} -  \rhob_{ij} \gbar_{jl} + \rhob_{jj} \gbar_{il} - \rhob_{jl} \gbar_{ij} \\
&= \rhob_{il} \gbar_{jj}.
\end{align*}
Once again we compute
\[ \rhob_{il} = \partial_i \partial_l \rho - \Gambar^s_{il} \partial_s \rho = - \Gambar^0_{il} = 0,\]
since the metric is diagonal and $\gbar_{00}$ is constant.  Therefore there are no contributions from tensors of type $A$ to the curvature tensor in the second case.
In light of these estimates
\begin{equation} \label{eqn:ah1est}
(R + \cK)_{ijkl} = O(\rho^{-3}),
\end{equation}
which implies \eqref{AH1}.

We now estimate the first covariant derivative of curvature.  The formula for components of covariant derivatives may be represented
\[ \nabla R = \nabla (R + \cK) = \partial (R+\cK) - \Gamma * (R+\cK) \]
where $*$ indicates contractions whose precise formula are unimportant here.

Estimates \eqref{eqn:wc} and \eqref{eqn:ah1est} show that $\Gamma * (R+\cK) = O(\rho^{-4})$.  For the derivative terms, since the metric depends only on $\rho$ and $y$, we need only compute those derivatives.  For the $\rho$ derivatives we have
\begin{align*}
 \partial_{\rho} (R + \cK)_{ijkl} &= -3 \rho^{-4} ( \overline{\nabla}^2 \rho \owedge \gbar)_{ijkl} + \rho^{-3}  \partial_{\rho}( \overline{\nabla}^2 \rho \owedge \gbar)_{ijkl} \\ &-2 \rho^{-3} \overline{R}_{ijkl} + \rho^{-2} \partial_{\rho} \overline{R}_{ijkl}, 
\end{align*}
By the estimates already established, each term is $O(\rho^{-4})$.

For the $y$-derivatives we find 
\begin{align*}
 \partial_{y} (R + \cK)_{ijkl} &= \rho^{-3} \partial_{y}( \overline{\nabla}^2 \rho \owedge \gbar)_{ijkl} + \rho^{-2} \partial_{y} \overline{R}_{ijkl}, 
\end{align*}
Already estimates \eqref{eqn:wc} are enough to handle the last term and when the $y$-derivative acts on the $\overline{\nabla}^2 \rho$ terms.  When the $y$-derivative acts on $\gbar$, by \eqref{eqn:gy} we obtain terms of order $O(\rho^{-3} \log \rho)$, which clearly are order $O(\rho^{-4})$.  Consequently
\[ \nabla_m R_{ijkl} = O(\rho^{-4}),\]
which implies \eqref{AH2e}.

Finally as $\gbar$ is a continuous metric on $\overline{\bU}$ it follows that $g$ is complete.  Set $r = -\log \rho$, and note from estimate \eqref{eqn:grho} that the second $g$-covariant derivatives of $r$ satisfy
\[ r_{\alpha \beta} = \rho^{-2} \gbar_{\alpha \beta} + O(\rho^{-1}). \]
From this estimate it follows that there exists $\rho_0$ sufficiently small where $r$ is strictly convex on the collar neighbourhood $r \geq - \log \rho_0$.  Further the hypersurface $r = -\log \rho_0$ is compact and strictly convex with respect to the outward unit normal.  From these facts we may deduce that $r \leq -\log \rho_0$ is totally convex and that the outward normal exponential map off $r = -\log \rho_0$ is a diffeomorphism onto its image.

\bibliographystyle{amsalpha}
\bibliography{lipah}
\end{document}